\documentclass[11pt]{amsart}
\usepackage{xypic,amscd,amsmath,amsthm,amssymb,textcomp}
\usepackage[dvips]{color}
\newtheorem{theorem}{Theorem}[subsection]
\newtheorem{lemma}[theorem]{Lemma}

\newtheorem{proposition}[theorem]{Proposition}
\newtheorem{corollary}[theorem]{Corollary}

\newcommand{\Br}{\mathrm{Br}}

\newcommand{\Pic}{\mathrm{Pic}}

\newcommand{\PP}{\mathbb{P}}

\newcommand{\QQ}{\mathbb Q}
\newcommand{\CC}{\mathbb{C}}

\newcommand{\ZZ}{\mathbb{Z}}
\newcommand{\FF}{\mathbb{F}}

\newcommand{\n}{\noindent}
\newcommand{\av}{\alpha _{van}}

\theoremstyle{definition}
\newtheorem{definition}[theorem]{Definition}

\theoremstyle{remark}
\newtheorem{remark}[theorem]{Remark}

\begin{document}

\date{}
\title[]
{Invariants of Vanishing Brauer Classes}

\author{Federica Galluzzi}
\address{Dipartimento di Matematica, Universit\`a di Torino, Via Carlo Alberto 10, Torino, Italy}
\email{federica.galluzzi@unito.it}
\author{Bert van Geemen}
\address{Dipartimento di Matematica, Universit\`a di Milano, Via Saldini 50, Milano, Italy}
\email{lambertus.vangeemen@unimi.it}


\begin{abstract}
A specialization of a $K3$ surface with Picard rank one to a $K3$ with rank two defines
a vanishing class of order two in the Brauer group of the general  $K3$ surface.
We give the $B$-field invariants of this class.
We apply this to the $K3$ double plane defined by a cubic fourfold with a plane. The specialization of such a cubic fourfold 
whose group of codimension two cycles has rank two to one which has rank three induces such a specialization of the double planes.
We determine the Picard lattice of the specialized double plane as well as the vanishing Brauer class and its relation to the natural `Clifford' Brauer class. 
This provides more insight in the specializations. It allows us to explicitly determine the $K3$ surfaces
associated to infinitely many of the conjecturally rational cubic fourfolds obtained as such specializations.
\end{abstract}

\maketitle



\section*{Introduction}
In this paper, $S$ will be a complex projective $K3$ surface.
An element $\alpha$ in the Brauer group $\Br(S)$ defines $\alpha$-twisted sheaves on $S$ which generate a
twisted derived category (\cite{HuySt05}).
A locally free $\alpha$-twisted sheaf of rank $n$ defines a projective space bundle over $S$.
Conversely, a projective bundle over $S$ is the projectivization of an $\alpha$-twisted locally free sheaf.
The two-torsion subgroup $\Br(S)_2$ allows one to describe the conic bundles over $S$ that are the exceptional divisors in $K3^{[2]}$-type hyperk\"ahler manifolds (\cite{vGK23}).

\n
A class $\alpha\in \Br(S)$ also defines a Hodge substructure $T_\alpha(S)$ of the transcendental lattice $T(S)$.
In the case that $S$ is a general $K3$ surface of degree two and $\alpha\in \Br(S)_2$ is a non-trivial class,
the Hodge substructure $T_\alpha(S)$
is Hodge isometric to either the transcendental lattice of a general cubic fourfold with a plane or
a general $K3$ surface of degree eight or to neither of these.

\n
In this paper we will be particularly concerned with specializations of a $K3$ surface $S$ and their impact on
$\Br(S)_2$.
That is, we consider a family of $K3$ surfaces over a disc with general fiber $S$ and special fiber $S_{2d}$.
We focus on the case where the rank of the Picard groups are one and two respectively,
see \cite{C02}, \cite{MT23} for more general cases.
Here the index $2d$ refers to the degree of the generator of $\Pic(S)$.
There are then natural identifications of the
second cohomology groups of the general and the special fiber. This easily implies that there is a restriction
map of Brauer groups from $\Br(S)_2$ to $\Br(S_{2d})_2$ which has a kernel of order two. The generator of this kernel
is called the vanishing Brauer class (of the specialization) and we denote it by $\av\in \Br(S)_2$,
$$
\langle\av\rangle\,=\,\ker\big(\Br(S)_2\,\longrightarrow\,\Br(S_{2d})_2\big)~.
$$
We work out the invariants of this Brauer class. In the case of a $K3$ of degree two these invariants
determine whether the Brauer class corresponds to a point of order two in $J(C)$,
where $C$ is the ramification curve defined by $S$, or to an even or odd theta characteristic on $C$.
These results were in a sense anticipated in the papers \cite{IOOV17}, \cite{Sk17} where the
restriction map $\Pic(S)\rightarrow \Pic(C)$ is studied in relation to $\Br(S)_2$.

\

\n
In the remainder of the introduction we discuss an application of vanishing Brauer classes to cubic fourfolds.
A well-known conjecture states that a (complex) cubic fourfold $X$ is rational if and only if it has
an associated $K3$ surface $S$, that is the transcendental lattice $T(X)\subset H^4(X,\ZZ)$ of $X$ is Hodge isometric to $T(S)(-1)$,
the transcendental lattice of $S$ with the opposite intersection form (\cite{Has00},\cite{Ku10},\cite{AdTh14}).
If it exists, $S$ is called a $K3$ surface associated to $X$. The general cubic fourfold $X'$ does not have an associated
$K3$ surface since $T(X')$ has rank $22$ whereas $T(S)$ has rank at most $21$.

\n
A much studied case is the one of a cubic fourfold $X$ containing a plane $P$. In that case $X$ defines a
$K3$ double plane $S=S_X$ with an odd theta characteristic, corresponding to a Brauer class $\alpha_X\in \Br(S_X)_2$.
In case $(X,P)$ is general, with group of codimension two cycles generated by the square of the hyperplane class and $P$, the $K3$ surface $S_X$ is \emph{not} an associated $K3$ surface since
$T(X)\not\cong T(S_X)(-1)$.
Instead there is a Hodge isometry $T(X)\cong T_{\alpha_X}(S_X)(-1)$, where $T_{\alpha_X}(S_X)$ is the index two sublattice defined by the class $\alpha_X$ (\cite{Vo86}).

\n
We consider now a specialization of a general $(X,P)$ to a fourfold where the group of algebraic
codimension two cycles $N^2(X)$ has rank three. These rank three lattices have been classified and
they are isomorphic to lattices $M_{\tau,n}$ for a pair $(\tau,n)$ of integers,
with $\tau\in\{0,\ldots,4\}$, $n\geq 2$, with a few cases that actually do not occur (\cite{YY23}).
We denote the specialization of $X$ by $X_{\tau,n}$ and the double plane it defines by
$S_{\tau,n}$. The $(\tau,n)$ such that $X_{\tau,n}$ has an associated $K3$ surface are given in
(\cite[Cor.\ 8.14]{YY23}).

\n
The specializiation $S_{\tau,n}$  of $S_X$ has Picard rank two,
hence this specialization defines a vanishing Brauer class $\av\in \Br(S_X)_2$.
In $\Br(S_X)_2$ we now have two Brauer classes, $\alpha_X$ and $\av$. A complete description of the
specialization from $S_X$ to $S_{\tau,n}$ requires taking into account not only the Picard lattice of $S_{\tau,n}$ and invariants of $\av$ but also the relation between $\alpha_X$ and $\av$. Our main application of vanishing Brauer classes, Theorem \ref{classifatn}, gives all this information. The rather long proof consists of explicit computations with lattices.

\n
A well-known case that we recover is the case that $\alpha_X=\av$ which was studied by Hassett (\cite{Has99}).
Then $\tau=1,3$ and $S_{\tau,n}$ \emph{is}
a $K3$ surface associated to $X_{\tau,n}$, and these are the only cases in which $S_{\tau,n}$ is
associated to an $X_{\tau,n}$. The quadratic surface bundle over $\PP^2$ defined by $(X_{\tau,n},P)$
then has a rational section and this implies that it is a rational fourfold. Since this fourfold is birational to
$X_{\tau,n}$, also $X_{\tau,n}$ is rational, thus verifying the conjecture.

\n
In case $\tau=0,4$, the vanishing Brauer class corresponds to a
point of order two on the ramification curve $C$ of the double plane $S_X$.
The sum $\beta_X:=\alpha_X+\av$ corresponds to a theta characteristic on $C$, which is even if and only if $n$ is odd.
A cubic fourfold $X_{\tau,n}$ has an associated $K3$ surface if and only if $n$ is odd.
It is of some interest to have a concrete description of this associated $K3$ surface.
Let $S_{\tau,n}$ be the $K3$ double plane defined by $X_{\tau,n}$ and let $C_{\tau,n}$ be the branch curve
of the double cover $S_{\tau,n}\rightarrow\PP^2$.
Let $\beta$ be the even theta characteristic on the branch curve $C_{\tau,n}$
which is the specialization of $\beta_X$ on $C$. Then $\beta$ defines a $K3$ surface $S_\beta$
which has a natural degree eight polarization.
In Proposition \ref{prop:Sbeta} we show that $S_\beta$ is a $K3$ surface associated to $X_{\tau,n}$.
We discuss the example in \cite{ABBV14}, which has $(\tau,n)=(4,5)$, in Section \ref{abbv}.

\

\section{Brauer Groups of $K3$ surfaces}\label{brk3}

\subsection{Brauer classes and B-fields}
Let $S$ be a $K3$ surface. The \textit{Brauer group} of $S$ is (cf.\ \cite[18.1]{Huy16})
$$\Br (S)= H^2(S, \mathcal{O}_S ^*)_{tors}~.
$$
The exponential sequence in this case gives
$$
0 \longrightarrow H^2(S,\ZZ)/ \Pic(S) \longrightarrow H^2(S,\mathcal O_S)\longrightarrow H^2(S,\mathcal O _S^*) \longrightarrow 0~.
$$ 
A two-torsion class $\alpha  \in \Br(S)_2$ has a lift $\tilde{\alpha}$ to the one dimensional complex vector space $H^2(S,\mathcal O_S)$ with $2 \tilde {\alpha} \in H^2(S, \ZZ)/\Pic(S)$.
Any class $B = B_{\alpha} \in \frac{1}{2}H^2(S,\ZZ)\subset H^2(S,\QQ) $ mapping to $\tilde {\alpha}$ is called a $B$-field representative of $\alpha$ (see \cite{HuySt05}).
A $B$-field $B_{\alpha}$ is unique
up to $(1/2)\Pic(S) + H^2(S,\ZZ)$ : 
$$
B_{\alpha}'=B_{\alpha}+\frac{1}{2}p+c,\qquad p\in \Pic(S) ,\quad c \in H^2(S,\ZZ)~,
$$
(see \cite[\textsection 4]{Huy05}, see \cite[\textsection 6]{Ku10}).
Assume now that $S$ is a general polarized $K3$ surface. There is the following

\begin{lemma}(\cite[Lemma 6.1]{Ku10}, \cite[Lemma 2.1]{vGK23}) \label{lem:invalpha}
Let $S$ be a $K3$  surface such that $\Pic(S)= \ZZ h$, $h^2=2d >0. \,$
Let $\alpha \in \Br(S)_2$ and $B_{\alpha} \in \frac{1}{2} H^2(S,\ZZ) \subset H^2(S,\QQ)$ a $B$-field representing $\alpha .\,$
The intersection numbers 
\begin{enumerate}
\item $B_\alpha h \mod \ZZ$,
\item  $B_\alpha^2 \mod \ZZ$, only in the case that $4B_\alpha h+h^2 \equiv 0 \mod 4$,
\end{enumerate}
are invariants of $\alpha$.
\end{lemma}

\subsection{Brauer Groups and $K3$ Lattices}\label{brgrlat}
There is an isomorphism, with $\rho(S)$ the Picard number of $S$,
$$
\Br(S) \cong \big(H^2(S, \ZZ)/\Pic(S)\big) \otimes \QQ / \ZZ\,\cong\,(\QQ/\ZZ)^{22-\rho(S)} ~.
$$ 
The lattice $H^2(S,\ZZ)$ is selfdual since it has a unimodular intersection form.
A class $\alpha \in \Br(S)_2$
can thus be identified with a homomorphism
$$
\alpha :\, T(S)\, \longrightarrow\, \ZZ / 2 \ZZ ~.
$$
If $B_\alpha$ represents $\alpha ,\,$ this homomorphism is given by
$$
\alpha:\,x\, \longmapsto\, x \cdot B_\alpha \mod \ZZ\quad(\mbox{in$\;\frac{1}{2}$}\ZZ/\ZZ)~.
$$
The class $\alpha$ defines a sublattice $T_\alpha(S) :=\ker(\alpha)$ in $T(S).$
For a $K3$ surface with Picard rank one the invariants of $\alpha$ given in Lemma \ref{lem:invalpha}
are invariants of the lattice $T_\alpha(S)$. In fact, index two sublattices $T_\alpha,T_\beta$ of $T(S)$ are isometric if and only if
$\alpha$ and $\beta$ have the same invariants  (\cite[Thm.2.3]{vGK23}).

\

\section{Vanishing Brauer Classes and Invariants}
\setcounter{subsection}{1}
\begin{definition}
Let $(S,h)$ be a general polarized $K3$ surface with $\Pic(S)= \ZZ h$, $h^2=2d>0.$
Consider a specialization
$S_ {2d}$  of $S$ where the Picard rank of $S_{2d}$ is two,
so $\Pic (S_{2d})= \ZZ h \oplus \ZZ k,$ for some divisor class $k$ which is primitive in $H^2(S_ {2d},\ZZ)$.
(By a specialization of $S$ we mean a family of quasi-polarized $K3$'s over a complex disc $\Delta$ such that
the fiber over $0\in \Delta$ is $S_{2d}$ and the fiber over some non-zero $a\in\Delta$ is $S$).
We may then identify
$$
H^2(S,\ZZ)\,=\,H^2(S_{2d},\ZZ)~,
$$
and we have inclusions
$$
\Pic(S)\,\subset\,\Pic(S_ {2d}),\qquad T(S)\,\supset\,T(S_ {2d}).
$$
Thus there is a restriction map $\Br(S)\rightarrow \Br(S_ {2d})$ given by restriction of the homomorphism $\alpha$
to $T(S_ {2d})$.
Since
$\Br(S)_2 \cong (\ZZ / 2\ZZ)^{21}$  and $\Br(S_{2d})_2 \cong (\ZZ / 2\ZZ)^{20}$,
there is a unique order two Brauer class that becomes trivial in $\Br(S_{2d})$, that is, $\alpha$ generates the kernel of the restriction map.
This class $\alpha_{van}$ is the \textit{vanishing Brauer class} (in this specialization).
\end{definition}

\begin{proposition} \label{prop:Brepav}
A $B$-field representative of $\av$ is provided by $B=k/2\;(\in \frac{1}{2}H^2(S,\ZZ))$ .
\end{proposition}

\begin{proof}
Since $k/2\not\in (1/2)\Pic(S)+H^2(S,\ZZ)$, it defines a non-trivial class in $\Br(S)_2$. On the other hand,
obviously $k/2\in(1/2)\Pic(S_{2d})$, hence $k/2$ defines the trivial class in $\Br(S_{2d})_2$. Therefore
$B=k/2$ is a B-field representative of $\av$.
\end{proof}

\n
This allows us to read off the invariants of $\av$ from the intersection
matrix of  $\Pic(S_{2d})= \ZZ h \oplus \ZZ k$ which can be written as
\[
\begin{pmatrix}
h^2 & hk \\
hk & k^2 
\end{pmatrix}   = \begin{pmatrix}
2d & b \\
b & 2c 
\end{pmatrix}\quad \mbox{for some \;} b,c \in \ZZ ~.
\]
Using the B-field representative $k/2$ of $\av$ given in Proposition \ref{prop:Brepav}
one finds the corollary below.

\smallskip
\n
\begin{corollary}\label{cor:invB}
The invariants of $\av \in \Br(S)$ are
\[
B_{van} h\,\equiv\, (1/2)b \mod\ZZ,\quad B_{van}^2\,\equiv\,(1/2)c \mod\ZZ,  
\]
where $B_{van}$ is any B-field representing $\av$. 
In particular, $2B_{van} h\equiv \operatorname{disc}(\Pic(S_{2d}))\mod 2$.
\end{corollary}

\subsection{Theta characteristics and Brauer classes on double planes}
For a smooth curve $C$ of genus $g$ recall that $x\in \Pic(C)$ is a two-torsion point if $2x=0$ and it is a theta characteristic if $2x=K_C$, the canonical class of $C$.
A theta characteristic $L$ is called even/odd if $h^0(L)$
is even/odd, there are $2^{g-1}(2^g+1)$ even and $2^{g-1}(2^g-1)$ odd theta characteristics.
The parity of a theta characteristic does not change under a deformation of $C$.

\n
The two-torsion points are a group, denoted by $J(C)_2\,(=\Pic(C)_2)$;
the sum $p+L$ of a point of order two $p$
with a theta characteristic $L$ is again a theta characteristic, but the parity may change;
the sum $L+M$ of two theta characteristics can be written as $K_C+p$ for a unique two-torsion point $p$.
The union of the sets of two torsion points and theta characteristics thus has a group structure.
This group can be identified with the two-torsion group $(\Pic(C)/\langle K_C\rangle)_2$,
which has order $2^{2g+1}$.
For any double plane $S$ with smooth branch curve $C_6$, there is a surjective map
$$
{\mathcal A}:\,(\Pic(C_6)/\langle K_{C_6}\rangle)_2\, \longrightarrow\, \Br(S)_2
$$
which is thus an isomorphism if rank$(\Pic(S))=1$ (\cite[Thm 1.1]{IOOV17}, \cite[\textsection 9]{vG05}).
The kernel of this map is given by restrictions of certain line bundles on $S$ to $C_6$.

\n
We now consider the case where we specialize a double plane $S$, with Picard rank one,
to a $K3$ surface $S_{2}$ with Picard rank two. In this case
\begin{equation}\label{pics2}
\Pic(S_{2})= \left ( \ZZ h \oplus \ZZ k, \begin{pmatrix}
2 & b \\
b & 2c 
\end{pmatrix}
\right )~.
\end{equation}
The change of basis which fixes $h$ and maps $k\mapsto k-mh$ where $b=2m,2m+1$, shows that we may assume $b=0,1$.
We find the following characterization of the class $\av$ in terms of line bundles on  $C_6$.

\begin{proposition}\label{prop:corrBrC}
Let $(S,h)$ be a general double plane specializing to $S_2$ with
$\Pic (S_2)= \ZZ h \oplus \ZZ k$ and intersection matrix of the form (\ref{pics2}). Then,
\begin{enumerate}
\item If $B_{\av}h \equiv 0$ ($b$ is even), $\av$ corresponds to a point of order two $p \in Jac(C_6).$
\item If $B_{\av}h \equiv \frac{1}{2}\,$ and $\,B_{\av}^2 \equiv \frac{1}{2}$, ($b,c$ odd), $\av$  corresponds to an odd theta characteristic on $C_6.$
\item If $B_{\av}h \equiv \frac{1}{2}\, $ and $\,B_{\av}^2 \equiv 0$ ($b$ odd and $c$ even), $\av$ corresponds to an even theta characteristic on $C_6.$
\end{enumerate}
\end{proposition}

\begin{proof}
This follows from Corollary \ref{cor:invB} and \cite[\S 9]{vG05}, \cite{IOOV17}.
Notice that \cite[Theorem 1.1]{IOOV17} shows that the vanishing Brauer class is obtained
from the restriction of a line bundle on $S$ to the ramification curve $C_6\subset \PP ^2$.
\end{proof}

\smallskip
\n

\

\n
\subsection{Orbits}\label{thetas}
In this paper we are interested in the Brauer class $\alpha=\alpha_X\in Br(S)_2$ defined by a cubic fourfold
with a plane. This Brauer class corresponds to an odd theta characteristic on the branch curve $C_6$
which has genus $10$. We also consider a Brauer class $\av\in Br(S)_2$ which depends on a specialization.
Here we determine which pairs $(\alpha_X,\av)$ can occur up to a natural equivalence.

Let ${\mathcal U}\subset \PP(H^0(\PP^2,\mathcal O(6))\cong \PP^{27}$  be the open subset whose points define
smooth sextic curves, such a curve has genus ten. There is a finite unramified covering ${\tilde{\mathcal U}}\rightarrow{\mathcal U}$
of degree $2^{21}$ whose fiber over $C$ is the group $(\Pic(C)/\langle K_C\rangle)_2$. It follows from \cite{Be00}
that this covering has four connected components defined by the subsets: $\{0\}$, $J(C)_2$ and the sets of even theta and odd theta characteristics respectively.
In particular, the monodromy group of this covering, which is $Sp(20,\FF_2)$, acts transitively on the odd theta characteristics.

\begin{lemma}\label{lem:orbits}
The stabilizer of an odd theta characteristic $\alpha$ is an orthogonal subgroup $O^-(20,\FF_2)\subset Sp(20,\FF_2)$.
The orbits of this stabilizer on a fiber are well known:
\begin{enumerate}
\item $\{0\}$,
\item $\{p\in J(C)_2-\{0\}:\, p+\alpha\; \mbox{is odd}\}$,
\item $\{p\in J(C)_2-\{0\}:\; p+\alpha\; \mbox{is even}\}$,
\item $\{\alpha \}$,
\item $\{ \mbox{odd theta characteristics distinct from}\;\alpha  \}$,
\item $\{\mbox{the even theta characteristics}\}$.
\end{enumerate}
\end{lemma}

\begin{proof}
This can be deduced for example from the results of Igusa \cite[V.6]{Ig72}. The group $J(C)_2$,
with the Weil pairing, can be identified with a $2g$-dimensional vector space $P$ over $ \FF _2=\ZZ / 2 \ZZ$  with a bilinear alternating map $e:P\times P\rightarrow \{\pm 1\}$. For a theta characteristic $L$ on $C$ define the quadratic form $q_L:J(C_2)\rightarrow \FF_2$
by $q_L(p):=h^0(L+p)-h^0(L)\mod 2$ (\cite[Theorem 1.13]{Har82}).
The theta characteristics are then identified with the set $T$ of maps $c:P\rightarrow \FF _2$ such that
$c(r+s)=c(r)c(s)e(r,s)$ where $c(r)=(-1)^{q_L(r)}$.

\n
Then \cite[Corollary p.\ 213]{Ig72} states that the symplectic group defined by $e$ on $P$
is doubly transitive on both the even and on the odd theta characteristics. Thus the stabilizer of an
odd theta characteristic is transitive on the set of the remaining odd theta characteristics.
From \cite[Proposition 2]{Ig72} one deduces that the stabilizer of an odd theta characteristic is transitive on
the even theta characteristics. Since the second and third orbits (in $J(C)_2$) are in bijection with the
fifth and sixth orbits (in $T$) respectively, the Lemma follows.
\end{proof}

\

\section{Cubic fourfolds containing a plane}

\subsection{Cubics with a plane and $K3$ double planes}
Let $X$ be a smooth cubic hypersurface in $\PP ^5 (\CC)  \,$ containing a plane $P.\,$ 
Consider the projection from the plane $P$ onto a plane in $\mathbb P^5$ disjoint from $ P.\,$ Blowing up $X$ along $P$, one obtains a quadric surface bundle $\pi \ : \ Y \longrightarrow \mathbb P^2$. The rulings of the quadrics define a double cover $S=S_X$ of $\PP^2$ branched over a degree six curve $C _6,\,$ the discriminant sextic.
 If  $X$ does not contain a second plane intersecting $ P,\,$ the curve $C_6$ smooth and $S$ is a $K3$ surface
 (see \cite[\textsection 1 Lemme 2]{Vo86}). 

\begin{equation}\label{qbundle}
\xymatrix{
 & Y=Bl_P(X) \ar[dr]^{q} \ar[r]  & X \ar@{.>}[d] & \!\!\!\!\!\!\!\!\!\!\!\!\!\!\!\!\!\! \supset P \\ 
&  C_6 \ar@{^{(}->}^{i}[r]  & \mathbb P^2 &
}	
\;\;\;\;\;
\end{equation}
The rulings of the quadrics of the bundle also define a $\PP ^1$-bundle $F$ over $S$ which gives a Brauer class $\alpha _X \in \Br (S)_2$, also known as the Clifford class.  This class $\alpha _X$ corresponds to an odd theta characteristic
$L$ on $C_6$ with $h^0(L)=1$  (see \cite[\textsection 2]{Vo86}).

\n
Conversely, a smooth plane sextic with such an odd theta characteristic defines a cubic fourfold with a plane which is obtained as in loc.\ cit.\ and also from the minimal resolution of the push-forward of $L$ to $\PP^2$ as in \cite{Be00}. 

\smallskip
\n
\subsection{Lattices}The cohomology group $H^4(X,\ZZ)$ with the intersection form is a rank $23$ odd, unimodular, lattice of signature $(2+,21-)$. 
It is also a Hodge structure with Hodge numbers $h^{3,1}=1$, $h^{2,2}=21$.
Let $h_3\in H^2(X,\ZZ)$ be the class of a hyperplane section and let $h_3^2\in H^4(X,\ZZ)$ be its square.
Denote with $N^2(X) \subset H^4(X,\ZZ)$ the odd, positive definite,
lattice of classes of codimension two algebraic cycles.
The transcendental lattice of $X$  is the even lattice defined as
$$
T(X)\,:=\, N^2(X)^{\perp}\; \subset H^4(X,\ZZ)~.
$$
The following proposition follows from \cite{Vo86}.
\smallskip
\n
\begin{proposition}\label{TXisoTalpha}
Let $X$ be a smooth cubic fourfold with a plane and let $S$ be the $K3$ double plane defined by $X$.
Then there is a Hodge isometry:
$$
T(X)\,\cong\, T_{\alpha_{X}}(S)(-1)
$$ 
with $T_{\alpha _{X}}(S):= \ker \alpha _X:T(S)\rightarrow \ZZ/2\ZZ$
(it is a sublattice of index two in $T(S)$ if $\alpha_X$ is non-trivial). It follows that
$$
rank(N^2(X))\,=\,rank(Pic (S))\,+\,1.
$$
\end{proposition}

The general cubic fourfold $X$ with a plane $P$ has
$$
N^2(X)\cong K_8:=\,\left(\ZZ h_3^2\oplus \ZZ P,\,\,\begin{pmatrix} 3&1\\1&3\end{pmatrix}\,\right)~.
$$

\

\section{Noether-Lefschetz divisors in $\mathcal C_8$ and double planes}

\subsection{The divisors ${\mathcal C}_d$ in ${\mathcal C}$}
Hassett determined all Noether-Lefschetz divisors in the moduli space $\mathcal C$ of cubic fourfolds.
These (irreducible) divisors are denoted by $\mathcal C _d$
and $d > 6$, $d \equiv  0, 2 \mod 6$.
The divisor $\mathcal C _d$ parametrizes fourfolds $X$ with
a certain rank two sublattice, containing $h_3^2$, denoted by $K_d\subset N^2(X)$ where $d=\operatorname{disc}(K_d)$.
Thus $\mathcal C_8$ parametrizes the cubic fourfolds with a plane since then $K_8\subset N^2(X)$.

\subsection{The divisors ${\mathcal C}_M$ in ${\mathcal C_8}$}
Yang and Yu give a classification of all Noether-Lefschetz divisors in  $\mathcal C_8$,
that is the divisors that parametrize cubics $X$ with a plane and $\mbox{rank}N^2(X)>2$.
They correspond to positive definite saturated sublattices of rank three $M \subset H^4(X,\ZZ)$ with 
$K_8\subset M$, up to isometry.
Denote by $\mathcal C_M \subset \mathcal C_8$ the divisor of smooth cubic fourfolds $X$
with such an isometry class of embeddings $M\hookrightarrow H^4(X,\ZZ)$.

\smallskip
\n
\begin{proposition}\label{prop:Ataun}
(\cite[Corollary 8.14]{YY23})\label{prop:Mtaun}
Consider the pairs of integers $(\tau,n)$ such that $\tau =0,...,4,$   $n \geq 2$  and $(\tau,n)\neq (3,2),(4,2),(4,3).$
Let $M_{\tau,n}$ be the rank three positive definite lattice with intersection matrix given by
\[
A_{\tau,n} = \begin{pmatrix}
3 & 1 & 0 \\
1 &3 &\tau\\
0&\tau&2n
\end{pmatrix}~,\qquad M_{\tau,n}\,:=\,(\,\ZZ^3,\,A_{\tau,n}\,)~.
\]
\begin{enumerate}
\item If $M$ is a positive definite rank $3$ lattice such that $K_8\subset M$ and such that $M$ has a saturated embedding in $H^4(X,\ZZ)$ then $M\cong M_{\tau,n}$ with $(\tau,n)$ as above.
\item Up to isometry, there is a unique embedding $M_{\tau,n}\hookrightarrow L\cong H^4(X,\ZZ)$ such that the
first basis vector maps to $h_3^2$.
\item The divisor $\mathcal C_{M_{\tau,n}}$ in $\mathcal C_8$ is non-empty and irreducible.
Moreover $\mathcal C_{M_{\tau,n}}=\mathcal C_{M_{\tau ',n '}}$ if and only if $(\tau,n)=(\tau ',n ').$
\end{enumerate}
\end{proposition}

\subsection{The specializations of $X$ and $S$}
Let $X$ be a cubic fourfold with a plane with $N^2(X)=K_8$, which has rank two.
Let $S$ be the $K3$ double plane defined by $(X,P)$ with branch locus $C$ and Brauer class $\alpha_{X}$
which we identify with an odd theta characteristic on $C$.

\n
We specialize $X$ to $X_{\tau,n}$ when $N^2(X_{\tau,n})=M_{\tau,n}$, which has rank three.
Let $S_2=S_{\tau,n}$ be the $K3$ double plane defined by $X_{\tau,n}$, it has Picard rank two.
The specialization of cubic fourfolds defines a specialization of $K3$ surfaces
$S$ to $S_{\tau,n}$. Hence it defines a vanishing Brauer class $\av\in \Br(S)$.
This vanishing Brauer class is non-trivial and thus lies in exactly one of the orbits $(2)\ldots(6)$
of the stabilizer of $\alpha_{X}$.

\n
In the following theorem we determine the Picard lattice of the specialization $S_{\tau,n}=S_2$ of $S$
and we also determine the orbit of $\av$.

\

\

\n

\smallskip
\n
\begin{theorem}\label{classifatn}
Let $X$ be a general cubic fourfold with a plane, so with rank$\,N^2(X)=2$.
Let $S$ be the $K3$ double plane defined by $X$, let $C$ be the branch curve
and let $\alpha_X\in \Br(S)_2$ be the Clifford class.
Let $X_{\tau,n}$ be a specialization of $X$ such that
$$ 
N^2(X_{\tau,n})\,\cong\, M_{\tau,n}~.
$$
Let $S_{\tau,n}$ be the $K3$ double plane defined by $X_{\tau,n}$.
Then the specialization of $K3$ double planes from $S$ to $S_2=S_{\tau,n}$ has the following properties.
\begin{enumerate}
\item{$\tau = 0$ :}
The Picard lattice of $S_{0,n}$ is
$$
\Pic(S_{0,n})\,\cong\,\begin{pmatrix}
              2&0\\0&-2n
             \end{pmatrix}~.
$$
The Brauer class $\alpha _{van}$ corresponds to a point of order two $p \in Jac(C_6)$.
Moreover, the theta characteristic $p+\alpha _X$ is even/odd exactly when $n$ is odd/even,
\item{$\tau = 1$ :}
The Picard lattice of $S_{1,n}$ is
$$
\Pic(S_{1,n})\,\cong\,\begin{pmatrix}
              2&1\\1&2-8n
             \end{pmatrix}~.
$$
Moreover $\alpha _X = \av  $, so the two classes coincide.
\item{$\tau = 2$ :} The Picard lattice of $S_{2,n}$ is
$$
\Pic(S_{2,n})\,\cong\,\begin{pmatrix}
              2&1\\1&2-2n
             \end{pmatrix}~.
$$
Moreover, $\alpha _X \neq \av $ and $\av$ corresponds to a theta characteristic which is even/odd when
$n$ is odd/even.
\item{$ \tau = 3 $:} The Picard lattice of $S_{3,n}$ is
$$
\Pic(S_{3,n})\,\cong\,\begin{pmatrix}
              2&1\\1&14-8n
             \end{pmatrix}~.
$$
Moreover, $\alpha _X = \av  $, so the two classes coincide.
\item{$\tau = 4$ :}The Picard lattice of $S_{4,n}$ is
$$
\Pic(S_{4,n})\,\cong\,\begin{pmatrix}
              2&0\\0&6-2n
             \end{pmatrix}~.
$$
The Brauer class $\alpha _{van}$ corresponds to a point of order
two $p \in Jac(C_6)$.  The theta characteristic $p + \alpha _X$ is even/odd when $n$ is odd/even.
\end{enumerate}
\end{theorem}

\

\subsection{Remark}
Consider a $K3$ double plane with a Picard lattice $diag(2,2c)$ for some $c<0$.
Choose an odd theta characteristic with $h^0=1$ on the branch curve and let $N^2$ be the rank three lattice
of algebraic codimension two cycles on the associated cubic fourfold.
Then the theorem above shows that $N^2$ is isometric to either
$M_{0,c}$ or $M_{4,3-c}$. One needs information on the orbit of the vanishing Brauer class
of the specialization of a general double plane with Picard rank one to the $K3$ under consideration
to determine which of the two is the correct one.

\subsection{The proof of the main result} The remainder of this section is devoted to the proof of Theorem
\ref{classifatn}.
The unimodular odd lattice $H^4(X,\ZZ)$ is isometric to
$$
H^4(X,\ZZ)\,\cong\,L\,:=\,\langle 1\rangle^{\oplus 3}\,\oplus\,U\,\oplus \,\Lambda',\qquad
\Lambda'\cong U\oplus E_8^{\oplus 2}~,
$$
where $U=\left(\ZZ^2,\left(\begin{smallmatrix}0&1\\1&0\end{smallmatrix}\right) \right)$.
Up to isometry there is a unique primitive embedding $K_8\hookrightarrow L$, and we choose one.
Then we verify explicitly that $K_8^\perp$ is isometric to the lattice $T_\alpha(S)$ for a general K3 surface $S$
of degree two and an $\alpha\in Br(S)_2$.

Next we extend the embedding of $K_8$ to a primitive embedding of $M_{\tau,n}\hookrightarrow L$.
Then the perpendicular
$M_{\tau,n}^\perp$ in $L$ is isometric to $T(X_{\tau,n})=T_{\alpha}(S_{\tau,n})(-1)$.
Since this lattice is contained in $T({X})=T_{\alpha}(S)(-1)\subset\Lambda(-1)$, where $\Lambda$
is the K3-lattice, we have found the sublattice $T_\alpha(S_{\tau,n})\subset T_{\alpha}(S)\subset\Lambda$.
The perpendicular of $T_\alpha(S_{\tau,n})$ in $\Lambda$
is then $Pic(S_{\tau,n})$.

To compute $Pic(S_{\tau,n})$, we choose a non-zero vector $t_{\tau,n}\in M_{\tau,n}$ such that
$$
M_{\tau,n}\;\supset\; K_8\,\oplus_\perp\,\ZZ t_{\tau,n}~.
$$
Then $t_{\tau,n}\in T(X)$ but in the specialization from $X$ to $X_{\tau,n}$ this class becomes algebraic.
The lattice $Pic(S_{\tau,n})$ is then the primitive rank sublattice of $\Lambda$ which contains both $h$,
where $Pic(S)=\ZZ h$, and $t_{\tau,n}$.

\

\subsection{The sublattice $K_8\subset L$} \label{sec:K8inL}
An element in $L$ will be written as a triple $(x,y,z)$ with
$x\in \langle1\rangle^{\oplus 3}$, $y\in U$ and $z\in \Lambda'$.
We consider the primitive embedding
{\renewcommand{\arraystretch}{1.3}
$$
K_8\,\hookrightarrow\,L,\qquad \left\{\begin{array}{rcl}
h_3^2&\longmapsto&\big((1,1,1),\left(\begin{smallmatrix}0\\0\end{smallmatrix}\right),0)\big),\\       P&\longmapsto&\big((1,0,0),\left(\begin{smallmatrix}1\\1\end{smallmatrix}\right),0)\big).
                                      \end{array}\right.
$$
Then one computes that $(h_3^2)^\perp=A_2\oplus U\oplus \Lambda'$ is an even lattice and that
$$
K_8^\perp\,=\,\langle \kappa_1,\kappa_2,\kappa_3\rangle\,\oplus\, \Lambda'~,
$$
where
$$
\begin{array}{rcl}
\kappa_1&=&\big((1,-1,0),\left(\begin{smallmatrix}0\\-1\end{smallmatrix}\right),0)\big),\\
\kappa_2&=&\big((0,1,-1),\left(\begin{smallmatrix}0\\0\end{smallmatrix}\right),0)\big),\\
\kappa_3&=&\big((0,0,0),\left(\begin{smallmatrix}1\\-1\end{smallmatrix}\right),0)\big),
\end{array}
\qquad
(\kappa_i\cdot\kappa_j)\,=\,
\begin{pmatrix} 2&-1&-1\\-1&2&0\\-1&0&-2\end{pmatrix}~.
$$

\subsection{The isomorphism $K_8^\perp\cong T_\alpha(S)(-1)\subset\Lambda$}\label{Talphaodd}
Let $(S,h)$ be a $K3$ surface of degree $2$ with $\Pic(S)=\ZZ h$.
There is an isomorphism
$$
H^2(S,\ZZ)\,\stackrel{\cong}{\longrightarrow}\, \Lambda_{}\,:=\,
U^{\oplus 3}\oplus E_8(-1)^{\oplus 2}\,=\,U\oplus U\oplus\Lambda'(-1)~.
$$
Under the isomorphism, we may assume that
$$
h\,=\,\big(\left(\begin{smallmatrix}1\\1 \end{smallmatrix}\right), 
\left(\begin{smallmatrix}0\\0 \end{smallmatrix}\right),0 \big)~.
$$
The transcendental lattice of $S$ is then
$$
T(S)\,=\,h^\perp\,=\,
\big\langle
\left(\begin{smallmatrix}1\\-1 \end{smallmatrix}\right)
\big\rangle \,
\oplus \,U \,\oplus\, \Lambda'(-1)~.
$$

Let $\alpha\in \Br(S)_2$ be a Brauer class
defined by an odd theta characteristic on $C_6$, the branch curve of $\phi_h:S\rightarrow\PP^2$,
equivalently, $\alpha$ is defined by a B-field $B$ with $hB\equiv 1/2\mod \ZZ$ and $B^2\equiv 1/2\mod\ZZ$.
We choose this B-field as follows:
$$
B\,=\,B_\alpha\,=\,(1/2)\big(\left(\begin{smallmatrix}0\\1 \end{smallmatrix}\right),\,
\left(\begin{smallmatrix}1\\1 \end{smallmatrix}\right),0\big)~.
$$

The Brauer class $\alpha$ corresponds to the homomorphism,
again denoted by $\alpha$:
$$
\alpha:\,T(S)\,\longrightarrow\,\ZZ/2\ZZ,\qquad t\,\longmapsto (t,2B)\mod 2.
$$
Since the image of $((p,-p),(q,r),v)\in T(S)$ is $p+q+r\mod 2$, the index two (non-primitive)
sublattice $T_\alpha(S)=\ker(\alpha)$ of $T(S)$ is
$$
T_\alpha(S)\,:=\,
\ker(\alpha)\,=\,\langle
\gamma_1,\;
\gamma_2,\;
\gamma_3\,
\rangle\oplus\Lambda'(-1)~,
$$
where
$$
\begin{array}{rcl}
\gamma_1&:=&\big(\left(\begin{smallmatrix}1\\-1 \end{smallmatrix}\right),
\left(\begin{smallmatrix}0\\1 \end{smallmatrix}\right),0\big),\\
\gamma_2&:=&\big(\left(\begin{smallmatrix}0\\0 \end{smallmatrix}\right),
\left(\begin{smallmatrix}1\\-1 \end{smallmatrix}\right),0\big),\\
\gamma_3&:=&\big(\left(\begin{smallmatrix}0\\0 \end{smallmatrix}\right),
\left(\begin{smallmatrix}1\\1 \end{smallmatrix}\big),0\right),
\end{array} \qquad
(\gamma_i\cdot\gamma_j)\,=\,-\begin{pmatrix} 2&-1&-1\\-1&2&0\\-1&0&-2\end{pmatrix}~.
$$
}
Thus we found an isomorphism $K_8^\perp\cong T_\alpha(S)(-1)$, with $\kappa_i\mapsto \gamma_i$ and which
is the identity  $\Lambda'\mapsto\Lambda'(-1)$.
Since $T(X)\cong T_\alpha(S)(-1)$, the Clifford class is $\alpha_X=\alpha$.

\

\subsection{An embedding $M_{\tau,n}\hookrightarrow L$}
For the $(\tau,n)$ as in Proposition \ref{prop:Ataun} we explicitly find an $m_{\tau,n}\in L$ such that
$$
M_{\tau,n}\,\stackrel{\cong}{\longrightarrow}\,
\langle\, h_3,P,m_{\tau,n}\,\rangle\subset\; \langle 1\rangle^{\oplus 3}\,\oplus\,U\,\oplus\,\Lambda'(-1)\,=\,L
$$
is a primitive embedding, in particular:
$$
h_3^2m_{\tau,n}\,=\,0,\quad \,Pm_{\tau,n}\,=\tau,\quad    m_{\tau,n}^2\,=\,2n~.
$$
We already embedded $K_8$ in \S \ref{sec:K8inL}, as
$\Lambda'=U\oplus E_8^{\oplus 2}$ and the primitive vector $v_{n}:=(1,n)\in U$ has
$v_n^2=2n$, we can choose:
$$
m_{\tau,n}\,=\,\big((0,0,0),\left(\begin{smallmatrix}0\\ \tau \end{smallmatrix}\right),v_n\big)~.
$$

\n
Next we determine a primitive, non-zero, vector $t_{\tau,n}\in L$, not necessarily primitive, such that
$$
\ZZ t_{\tau,n}\,\subset\, M_{\tau,n}\cap K_8^\perp\,=\,M_{\tau,n}\cap \langle h_3,P\rangle^\perp~.
$$
Since a general element in $M_{\tau,n}$ is $x=ah_3^2+bP+cm_{\tau,n}$ and $xh_3^2=3a+b$, $xP=a+3b+c\tau$,
we can take
$$
t_{\tau,n}\,=\,\tau h_3^2\,-3\tau P\,+\,8m_{\tau,n}\,=\,-\tau(2\kappa_1\,+\,\kappa_2\,+\,3\kappa_3)\,+\,8v_n~,
$$
where we used that $t_{\tau,n}\in K_8^\perp=T(X)$ which is spanned by the $\kappa_i$ and $\Lambda'$.

\n
We determine the image of $t_{\tau,n}$ in $T(X)(-1)=T_\alpha(S)\subset L$.
Recall that the $\kappa_i$ map to the $\gamma_i$ and that now $v_n^2=-2n$:
$$
t_{\tau,n}\,=\,-\tau(2\gamma_1\,+\,\gamma_2\,+\,3\gamma_3)\,+\,8v_n\,\in\, T_\alpha(S)~.
$$
Using the explicit expressions for the $\gamma_i$ one finds
$$
t_{\tau,n}\,=\,
2\big(\tau\left(\begin{smallmatrix}-1\\1 \end{smallmatrix}\right),
\tau\left(\begin{smallmatrix}-2\\-2 \end{smallmatrix}\right),4v_{n}\big)\,\in
\,U^{\oplus 2}\,\oplus \,\Lambda'(-1)~.
$$

\

\subsection{The Picard lattice $\Pic(S_{\tau,n})$ and the Brauer class $\av$}
We compute the Picard group of $S_{\tau,n}$, the $K3$ surface associated to a cubic fourfold $X_{\tau,n}$ with
$N^2(X)=M_{\tau,n}$ 
as:
$$
\Pic(S_{\tau,n})\,=\,\langle h,\,\mbox{$\frac{1}{2}$}t_{\tau,n}\rangle_{sat},
$$
where we used that $t_{\tau,n}$ is divisible by two in
$U^{\oplus 2}\,\oplus \,\Lambda'(-1)\cong H^2(S_{\tau,n},\ZZ)$.

We will do so for each of the five cases for $\tau$ in the next sections.
We determine the
vanishing Brauer class $\av\in \Br(S)$ for the specialization of $S$ to $S_{\tau,n}$
induced by the one of a general cubic fourfold with a plane to one with $N^2(X)=M_{\tau,n}$.
The finer classification of $\av$ in terms of the orbits of the
stabilizer of the Brauer class (the Clifford invariant) 
$\alpha=\alpha_X\in \Br(S)$ defined by the cubic fourfold $X$ is also given.
Recall from \S \ref{Talphaodd} that $\alpha$ has B-field representative
$$
B_\alpha\,=\,\mbox{$\frac{1}{2}$}\big(\left(\begin{smallmatrix}0\\1 \end{smallmatrix}\right),
\left(\begin{smallmatrix}1\\1 \end{smallmatrix}\right),0\big)~,
$$
and that $\Pic(S)=\ZZ h$, with
$$
h\,=\,
\big(\left(\begin{smallmatrix}1\\1 \end{smallmatrix}\right),
\left(\begin{smallmatrix}0\\0 \end{smallmatrix}\right),0\big)~.
$$

\

\subsection{The case $\tau=0$}\label{tau=0}
The Picard lattice is:
$$
\Pic(S_{0,n})\,=\,\langle h,\,\left(\left(\begin{smallmatrix}0\\0 \end{smallmatrix}\right),
\left(\begin{smallmatrix}0\\0 \end{smallmatrix}\right),4v_{n}\right)\,\rangle_{sat}\,=\;
\langle h,v_n\rangle\,=\,
\begin{pmatrix} 2&0\\0&-2n\end{pmatrix}~.
$$
Notice that $\det(\Pic(S_{0,n}))=-4n$ whereas $\det(A_{0,n})=16\cdot n-3\cdot0^2=16n$.

\n
The invariants of the vanishing Brauer class are determined from Corollary \ref{cor:invB}.
The Gram matrix of $\Pic(S_{0,n})$ has $b=0$, $2c=-2n$. Hence the vanishing Brauer class
has invariant $B_{van}h=0$ and thus $B_{van}^2$ is not an invariant.
By Proposition \ref{prop:corrBrC} $\av$ corresponds to {a point of order two} (as $b=0$ is even).

\n
A B-field representing the vanishing Brauer class is obtained from the second basis vector of $\Pic(S_{0,n})$:
$$
B_{van}\,:=\,
\mbox{$\frac{1}{2}$}
\left(\left(\begin{smallmatrix}0\\0 \end{smallmatrix}\right),
\left(\begin{smallmatrix}0\\0 \end{smallmatrix}\right),v_{n}\right)~.
$$
Now we use the addition in the Brauer group. The sum of the Clifford class and the
vanishing Brauer class has a B-field representative given by
$$
B_\alpha+B_{van}\,=\, \mbox{$\frac{1}{2}$}\big(\left(\begin{smallmatrix}0\\1 \end{smallmatrix}\right),
\left(\begin{smallmatrix}1\\1 \end{smallmatrix}\right),v_{n}\big)~.
$$
This $B$-field has invariant $h\cdot(B_\alpha+B_{van})=\mbox{$\frac{1}{2}$}$, so the sum of the Brauer classes corresponds to a theta characteristic.
Using $v_{n}^2=-2n$ one finds
$$
(B_\alpha+B_{van})^2\,=\,\mbox{$\frac{1}{4}$}(0+2+v_{n}^2)=\mbox{$\frac{1}{2}$}(-n+1)~.
$$
Using Proposition \ref{prop:corrBrC} again, we find that
the theta characteristic is even when $n\equiv 1 \mod 2$ and it is odd otherwise.

\

\subsection{The case $\tau=1$}\label{tau=1}
The Picard lattice is (notice that $h+\mbox{$\frac{1}{2}$}t_{1,n}$ is divisible by $2$):
$$
\Pic(S_{1,n})\,=\,\langle h,
\,\big(\left(\begin{smallmatrix}-1\\1 \end{smallmatrix}\right),
\left(\begin{smallmatrix}-2\\-2 \end{smallmatrix}\right),4v_{n}\big)\,\rangle_{sat}\,=\,
\langle h,\,
\big(\left(\begin{smallmatrix}0\\1 \end{smallmatrix}\right),
\left(\begin{smallmatrix}-1\\-1 \end{smallmatrix}\right),2v_{n}\big)\,\rangle~.
$$
The last generator has norm $2+4v_{n}^2=2-8n$ and
the Gram matrix of the Picard lattice w.r.t.\ this basis is
$$
\Pic(S_{1,n})\,=\,\begin{pmatrix} 2&1\\1&2-8n\end{pmatrix}~.
$$
Notice that $\det(\Pic(S_{1,n}))=3-16n$ which is the opposite of $\det(A_{1,n})$.
The Gram matrix has $b=1$, $2c=2-8n$. By Proposition \ref{prop:corrBrC} $\av$ corresponds to
a theta characteristic (as $b$ is odd) which is odd since $c=1-4n$ is odd.

\n
A B-field representing the vanishing Brauer class is obtained from the second basis vector of $\Pic(S_{1,n})$:
$$
B_{van} \,:=\,\mbox{$\frac{1}{2}$}
\big(\left(\begin{smallmatrix}0\\1 \end{smallmatrix}\right),
\left(\begin{smallmatrix}-1\\-1 \end{smallmatrix}\right),2v_n\big)
\,\equiv \,
\mbox{$\frac{1}{2}$}
\big(\left(\begin{smallmatrix}0\\1 \end{smallmatrix}\right),
\left(\begin{smallmatrix}1\\1 \end{smallmatrix}\right),0\big)\,=\,B_\alpha~,
$$
where the congruence is modulo 
$\mbox{$\frac{1}{2}$}\Pic(S_{1,n})\,+\,H^2(S_{1,n},\ZZ)$,
hence the vanishing Brauer class coincides with the one defined by $B_\alpha$, which is the Clifford class $\alpha_X$.

\

\subsection{The case $\tau=2$}\label{tau=2}
The Picard lattice is ($2h+\mbox{$\frac{1}{2}$}t_{2,n}$ is divisible by $4$):
$$
\Pic(S_{2,n})\,=\,\big\langle h,\,\big(\left(\begin{smallmatrix}-2\\2 \end{smallmatrix}\right),
\left(\begin{smallmatrix}-4\\-4 \end{smallmatrix}\right),4v_n\,\big)\rangle_{sat}
\,=\,
\langle h,\,
\big(\left(\begin{smallmatrix}0\\1 \end{smallmatrix}\right),
\left(\begin{smallmatrix}-1\\-1 \end{smallmatrix}\right),v_n\big)\,\rangle~,
$$
so the Gram matrix is:
$$
\Pic(S_{2,n})\,=\,\begin{pmatrix} 2&1\\1&2-2n\end{pmatrix}~.
$$
Notice that $\det(\Pic(S_{2,n}))=3-4n$ whereas $\det(A_{2,n})=16\cdot n-3\cdot 2^2=4(4n-3)$.
The Gram matrix of $\Pic(S_{2,n})$ has $b=1$, $2c=2-2n$. Hence the vanishing Brauer class is defined by a
{theta characteristic} (as $b$ is odd) which is {even/odd} iff $c=1-n$ is even/odd iff $n$ is odd/even.

\n
A B-field representing the vanishing Brauer class is obtained from the second basis vector of $\Pic(S_{2,n})$:
$$
B_{van}\,:=\,\mbox{$\frac{1}{2}$}
\big(\left(\begin{smallmatrix}0\\1 \end{smallmatrix}\right),
\left(\begin{smallmatrix}-1\\-1 \end{smallmatrix}\right),v_n\big)~.
$$
Since $v_n=(1,n)\in U$, we see that $B_{van}\not\equiv B_\alpha$, hence $\av\neq \alpha_X$.

\

\subsection{The case $\tau=3$}\label{tau=3}
The Picard lattice is (notice that $3h+\mbox{$\frac{1}{2}$}t_{3,n}$ is divisible by $2$):
$$
\Pic(S_{3,n})\,=\,\langle h,\,
\big(\left(\begin{smallmatrix}-3\\3 \end{smallmatrix}\right),
\left(\begin{smallmatrix}-6\\-6 \end{smallmatrix}\right),4v_n
\big)\,\rangle_{sat}\,=\,
\langle  h ,\, \big(
\left(\begin{smallmatrix}0\\3 \end{smallmatrix}\right),
\left(\begin{smallmatrix}-3\\-3 \end{smallmatrix}\right),2v_n\big)\,\rangle~,
$$
as the last vector has length $2\cdot3^2-4v_{n}^2=18-8n$ we find the Gram matrix:
$$
\Pic(S_{3,n})
\,=\,
\begin{pmatrix} 2&3\\3&18-8n\end{pmatrix}
\,\sim\,
\begin{pmatrix} 2&1\\1&14-8n\end{pmatrix}~,
$$
where the equivalence is given by replacing $e_2\mapsto e_2-e_1$ where the $e_i$ are the standard basis vectors.
Notice that $\det(\Pic(S_{3,n}))=27-16n$ which is the opposite of $\det(A_{3,n})$.
The Gram matrix of $\Pic(S_{3,n})$ has $b=1,2c=14-8n$. Hence the vanishing Brauer class corresponds
to a theta characteristic (as $b$ is odd) which is odd since $c=7-4n$ is odd.

\n
A B-field representing the vanishing Brauer class is obtained from the second basis vector of $\Pic(S_{3,n})$:
$$
B_{van}\,:=\,\mbox{$\frac{1}{2}$}
\left(\left(\begin{smallmatrix}0\\3 \end{smallmatrix}\right),
\left(\begin{smallmatrix}-3\\-3 \end{smallmatrix}\right),2v_n\right)\, \equiv \,
\mbox{$\frac{1}{2}$}
\left(\left(\begin{smallmatrix}0\\1 \end{smallmatrix}\right),
\left(\begin{smallmatrix}1\\1 \end{smallmatrix}\right),0\right)
\,\equiv\,B_\alpha\mod H^2(S_{3,n},\ZZ)~,
$$
hence the vanishing Brauer class coincides, as in the case $\tau=1$, with $\alpha_X$.

\subsection{The case $\tau=4$}\label{tau=4}
The Picard lattice is (notice that $\mbox{$\frac{1}{2}$}t_{4,n}$ is divisible by $4$):
$$
\Pic(S_2)\,=\,\langle h,\,t_{4,n}\,=\,
\big(\left(\begin{smallmatrix}-4\\4 \end{smallmatrix}\right),
\left(\begin{smallmatrix}-8\\-8 \end{smallmatrix}\right),4v_{n}\big)\,\rangle_{sat}\,=\,
\langle h,\,
\big(\left(\begin{smallmatrix}-1\\1 \end{smallmatrix}\right),
\left(\begin{smallmatrix}-2\\-2 \end{smallmatrix}\right),v_{n}\big)\,\rangle~.
$$
Hence
$$
\Pic(S_{4,n})\,=\,
\begin{pmatrix}
              2&0\\0&6-2n
             \end{pmatrix}~.
$$
Notice that $\det(\Pic(S_{4,n}))=12-4n$ whereas $\det(A_{4,n})=16\cdot n-3\cdot4^2=4\cdot (4n-12)$.
The vanishing Brauer class is defined by a { point of order two} (because $b=0$ is even).

\n
A B-field representing the vanishing Brauer class is obtained from the second basis vector of $\Pic(S_{4,n})$:
$$
B_{van}\,:=\,
\mbox{$\frac{1}{2}$}
\big(\left(\begin{smallmatrix}-1\\1 \end{smallmatrix}\right),
\left(\begin{smallmatrix}-2\\-2 \end{smallmatrix}\right),v_{n}\big)\,\equiv\,
\mbox{$\frac{1}{2}$}
\big(\left(\begin{smallmatrix}1\\1 \end{smallmatrix}\right),
\left(\begin{smallmatrix}0\\0\end{smallmatrix}\right),v_{n}\big)
$$
Notice that
$$
B_\alpha+B_{van}\,=\,\mbox{$\frac{1}{2}$}\big(
\left(\begin{smallmatrix}0\\1 \end{smallmatrix}\right),
\left(\begin{smallmatrix}1\\1 \end{smallmatrix}\right),0\big)\,+\,
\mbox{$\frac{1}{2}$}
\big(\left(\begin{smallmatrix}1\\1 \end{smallmatrix}\right),
\left(\begin{smallmatrix}0\\0 \end{smallmatrix}\right),v_{n}\big)\,\equiv\,
\mbox{$\frac{1}{2}$}\big(
\left(\begin{smallmatrix}1\\0 \end{smallmatrix}\right),
\left(\begin{smallmatrix}1\\1 \end{smallmatrix}\right),v_{n}\big)
$$
is a B-field with invariants
$h\cdot (B_\alpha+B_{van})=\mbox{$\frac{1}{2}$}$, so it corresponds to a theta characteristic, and
$$
(B_\alpha+B_{van})^2\,=\,\mbox{$\frac{1}{4}$}(0+2+v_n^2)\equiv \mbox{$\frac{1}{4}$}(2-2n)\equiv
\mbox{$\frac{1}{2}$}(1-n)
\,\mod\,\ZZ~.
$$
The theta characteristic corresponding to $\alpha_X+\av$ is thus even iff 
$n$ is odd.
\qed

\subsection{Remark} In \ref{tau=0},\ldots,\ref{tau=4} of the proof of Theorem \ref{classifatn}
we observed that $-4\det(Pic(S_{\tau,n}))=\det(M_{\tau,n})$ for $\tau=0,2,4$ whereas $-\det(Pic(S_{\tau,n}))=\det(M_{\tau,n})$ if $\tau=1,3$. This relation was already observed in
\cite[Proposition 1]{ABBV14}. Notice also that $\det(Pic(S_{\tau,n}))$ is even iff $\tau=0,4$ and that
in these cases $\alpha_X\neq \av$, so $\alpha_X$ restricts to a non-trivial Brauer class on $S_{\tau,n}$.
This was already shown in \cite[Proposition 2]{ABBV14}. See also \cite[Theorem 4.8 and Proposition 4.10]{Gal17}.

\

\section{Associated $K3$ surfaces and the divisors ${\mathcal C}_M$}

\subsection{Classification of admissible lattices in $\mathcal C_8$}
In \cite{YY23} there is also a lattice-theoretic characterization of the cubic fourfolds in $\mathcal C_8$
with an associated $K3$ surface, that is,
of those that are conjecturally rational.

\begin{definition}(\cite[Definition 8.1]{YY23})
A lattice $M_{\tau,n}$ is admissible if $T(X_{\tau,n})(-1)$ is Hodge isometric to the transcendental lattice
of a $K3$ surface.
\end{definition}

\begin{remark}
The definition in \cite{YY23} is different, but it is equivalent.
\end{remark}

\begin{proposition}(\cite[Corollary 8.14]{YY23})
The lattice $M_{\tau,n}$ is admissible if and only if one of the following conditions is true
\begin{enumerate}
\item{(a)} $\tau = 1, 3$ ;
\item{(b)} $\tau = 0, 2, 4$ and $n$ is odd.
\end{enumerate}
\end{proposition}

\smallskip
\n
We already discussed the cases $\tau=1,3$ in the introduction. In these cases $\av=\alpha_X$ and $S_{\tau,n}$,
the $K3$ double plane defined by $X_{\tau,n}$,
is a $K3$ surface associated to $X_{\tau,n}$.
Moreover, Hassett proved the rationality of these cubic fourfolds in \cite{Has99}.

The case $\tau=2$, $n$ odd, is still under investigation. In the remaining cases we have identified
an associated $K3$ surface, see the proposition below. We are investigating its geometry in relation to the cubic fourfolds.

\smallskip
\n
\begin{proposition}\label{prop:Sbeta}
Let $X_{\tau,n}$ be a cubic fourfold with $N^2(X_{\tau,n})\cong M_{\tau,n}$.
Let $S_{\tau,n}$ be the $K3$ double plane defined by $X_{\tau,n}$ and let $C_{\tau,n}$ be the branch curve
of the double cover $S_{\tau,n}\rightarrow \PP^2$.
Let $\tau =0,4 $ and $n$ odd, and
let $\beta$ be the even theta characteristic $\beta$ on $C_{\tau,n}$ which is the specialization
of the even theta characteristic $\beta_X:=\av + \alpha_{X}$.
Then the $K3$ surface $S_\beta$, a degree $8$ surface in $\mathbb P^5$,
defined by the even theta characteristic $\beta$ is a $K3$ surface associated to $X_{\tau,n}$.
\end{proposition}

\begin{proof}
Let $X$ be a general cubic fourfold with a plane and let $C_6$ be the branch curve of $S_X\rightarrow\PP^2$.
We proved that for $\tau =0,4$ the vanishing Brauer class $\av$
corresponds to a point of order two in $J(C_6)$ and that
$\beta_X := \av + \alpha _X$ corresponds to an even theta characteristic on $C_6$.
Specializing $C_6$ to $C_{\tau,n}$, we obtain an even theta characteristic $\beta$ on $C_{\tau,n}$.
Then (the push forward to $\PP^2$ of) $\beta$ admits a minimal resolution
\[
0 \longrightarrow \mathcal O_{\mathbb P^2}(-2)^6 \overset {M}{\longrightarrow}
\mathcal O_{\mathbb P^2}(-1)^6 \longrightarrow \beta \longrightarrow 0
\]
where $M$ is a $6 \times 6$ matrix of linear forms on $\mathbb P^2 \,$ and $\det M=F$ where
$F=0 $ is an equation defining $C_{\tau,n}$
(\cite[Proposition 4.2]{Be00}).
The base locus of these quadrics is a $K3$ surface $S_\beta$ of degree $8$ in $\mathbb P^5.\,$

\n
The transcendental lattice of $S_\beta$ is $T_\beta(S_{\tau,n})$, see \cite{vG05}, \cite{IOOV17}
for the case of a $K3$ with Picard rank one, and by specialization it also holds for $S_{\tau,n}$.
Since $\av$ is trivial on $T({S_{\tau,n}})$,
the homomorphisms $\beta$ and $\alpha_X$ have the same restriction to $T({S_{\tau,n}})$, hence
$$
T(S_\beta)\,\cong\,\ker(\beta:T({S_{\tau,n}})\rightarrow\ZZ/2\ZZ)\,=\,
\ker(\alpha_X:T({S_{\tau,n}})\rightarrow\ZZ/2\ZZ)\,=\,
T(X_{\tau,n})(-1)~.
$$
Therefore the $K3$ surface $S_\beta$ is associated to $X$.
\end{proof}

\subsection{Pfaffian cubic fourfolds with a plane}\label{abbv}
\n
An interesting example of a pfaffian, hence rational,
cubic fourfold $X$ with a plane and rank$(N^2(X))=3$, but with $\alpha_X\neq 0$, is
given in \cite[\textsection 4]{ABBV14}. They determine $\tau,n$ explicitly
(but notice that they use a different convention for writing the lattices $M_{\tau,n}$).
We verify that $(\tau,n)=(4,5)$ here, using our Theorem \ref{classifatn}.

\n
The double plane $S=S_X$ is branched along a smooth sextic $C=C_6 \subset \mathbb P^2$ with a tangent conic.
Since the inverse image of this conic consists of two smooth rational curves $n,n'$ in $S$,
the Picard lattice of $S$  is
$$
\Pic(S)\,=\,\left(\ZZ h\oplus\ZZ n,\,
\begin{pmatrix}
2&2\\2&-2
\end{pmatrix}\right)
\,\cong\,
\left( \ZZ h\oplus \ZZ (h-n),
\begin{pmatrix}
2&0\\0&-4
\end{pmatrix}\right)~.
$$
So we are in case $\tau=0,n=2$ or in the case $\tau =4, \, n=5$.
The vanishing Brauer class is thus a point of order two in $J(C)_2$ and $\alpha_X+\av$ is an odd/even
theta characteristic in the first/second second case respectively.

\n
The tangent conic cuts out a divisor $2D$ on $C\subset\PP^2$ and the rational curves $n,n'$ each cut out
$D$ on $C\subset S$. The intersection of $h$ with $C\subset S$ is a divisor class $D_l$ such  that
$2D_l\equiv 2D\in \Pic(C)$. The image of $h-n\in \Pic(S)$ is then $p:=D_l-D\in \Pic(C)$,
which  a point of order two, and $p$ must correspond to the vanishing Brauer class by \cite[Theorem 1.1]{IOOV17}.

\n
The Clifford class $\alpha_X$ corresponds to a theta characteristic $L$ on $C$ with $h^0(L)=1$ and $2L\in|K_C=3D_l|$ is cut out by a cubic curve $C_3$ tangent to $C$.
Following \cite{Vo86} and using the description of $X$ from \cite[\S 4]{ABBV14}, one finds that
$$
C_3:\quad 14x^3 + 15x^2y + 4x^2z + 9xyz + 14xz^2 + 16y^3 + 11y^2z + 8yz^2 + z^3\,=\,0
$$
is the determinant of the $3\times 3$ submatrix of linear forms in the $4\times 4$ matrix (obtained from the minimal resolution of $L$) defining the quadric
bundle on the cubic fourfold (\cite[Proposition 4.2b]{Be00}).
So if (with $x=x_0,\ldots,w=u_2$)
$$
\begin{array}{l}
F(x_0,x_1,x_2,u_0,u_1,u_2)\,=\,\\
(x_0-4x_1-x_2)u_0^2\,+\,\ldots\,-\,x_2^3\,=\,\sum_{0\leq i,j\leq 2} F_{ij}(x_0,x_1,x_2)u_iu_j\,+\,\ldots
\end{array}
$$
is the defining (pfaffian) cubic polynomial for $X$, where the remaining terms are of degree at most one in the $u_i$, then $C_3$ is defined by $\det(F_{ij})_{i,j=0,1,2}=0$.
The theta characteristic $L+p=L+D_q-D_l$ has, using Serre duality on $C$:
$$
h^0(L+p)\,=\,h^0(K_C-(L+D_p-D_l))\,=\,h^0(4D_l-(L+D_p))~.
$$
Since $|4D_l|$ is cut out by degree four curves on $C$ and an explicit (Magma) computation shows that there
are no such curves passing through the support of $L+D_p$, we conclude that $h^0(L+p)=0$, so $L+p$ is an even
theta characteristic.

\n
Comparing with Theorem \ref{classifatn} we see that we are indeed in the case that $(\tau,n)=(4,5)$,
that is $N^2(X)\cong M_{4,5}$.

\n
In particular, there is a $K3$ surface associated to $X=X_{4,5}$ which is identified in \cite{ABBV14}
with the pfaffian $K3$ surface associated by Beauville and Donagi in \cite{BD85}
to the pfaffian fourfold $X$. This $K3$ surface has a natural embedding of degree $14$ in $G(2,6)$.
Instead, we associated the degree eight $K3$ surface $S_\beta$ to $X$ in Proposition \ref{prop:Sbeta}.
In a sequel to this paper we intend to investigate $S_\beta$ in this particular case as well as for
$\tau=4$ and any odd $n$.

 \section{Availability of data and materials}
 
 Not applicable.
 
 \section{Competing interests}
 
 Not applicable.
 
 \section{Funding}
 
Federica Galluzzi was supported by the Italian Ministry of University and Research through the PRIN project n. 2022L34E7W.

\bibliographystyle{amsplain}

\end{document}